\documentclass[reqno, 11pt]{amsart}

\author[]{Yen-chi Roger Lin}
\address{Department of Mathematics, National Taiwan Normal University, Taipei City 11677, Taiwan}
\email{yclin@math.ntnu.edu.tw}

\author[]{Akihiro Munemasa}
\address{Graduate School of Information Sciences,
Tohoku University, Aoba Ward, Sendai, 980-8579, Japan}
\email{munemasa@tohoku.ac.jp}

\author[]{Tetsuji Taniguchi}
\address{Department of Electronics and Computer, Hiroshima Institute of Technology, Saeki Ward, Hiroshima, 731-5143, Japan}
\email{t.taniguchi.t3@cc.it-hiroshima.ac.jp}

\author[]{Kiyoto Yoshino}
\address{Department of Computer Science, Faculty of Applied Information Science, Hiroshima Institute of Technology, Saeki Ward, Hiroshima, 731-5143, Japan}
\email{k.yoshino.n9@cc.it-hiroshima.ac.jp}

\thanks{The research of A.~Munemasa was supported by JSPS KAKENHI Grant Number 20K03527.
The research of T.~Taniguchi was supported by JSPS KAKENHI Grant Number 21K03344.
The research of Y.~R.~Lin is partially supported by Taiwan NSTC grant 113-2115-M-003-016-MY2.}

\usepackage{amsmath,amsthm,amsfonts,color,comment}
\usepackage[T1]{fontenc}
\usepackage{times}
\usepackage[showonlyrefs]{mathtools}
\usepackage[]{hyperref}
\usepackage[shortlabels]{enumitem}
\setlist[enumerate,1]{label={\upshape(\roman*)}}
\usepackage{hyperref}

\newcommand{\nexteq}{\displaybreak[0]\\ &=}

\newcommand{\refby}[1]{&&\text{(by (\ref{#1}))}}

\numberwithin{equation}{section}

\newtheorem{lem}{Lemma}[section]
\newtheorem{thm}[lem]{Theorem}%[section]
\newtheorem{cor}[lem]{Corollary}%[section]

\theoremstyle{definition}
\newtheorem{dfn}[lem]{Definition}%[section]
\newtheorem{remark}[lem]{Remark}%[section]
\newtheorem{example}[lem]{Example}%[section]

\DeclareMathOperator{\Aut}{Aut}
\DeclareMathOperator{\disc}{disc}
\DeclareMathOperator{\vol}{vol}
\DeclareMathOperator{\id}{id}
\DeclareMathOperator{\sw}{sw}
\DeclareMathOperator{\iso}{iso}

\newcommand{\cY}{\mathcal{Y}}
\newcommand{\cI}{\mathcal{I}}
\newcommand{\aes}{\omega}
\newcommand{\aeso}{\phi}
\newcommand{\aesf}{\psi}
\newcommand{\cX}{\mathcal{X}}
\newcommand{\cF}{\mathcal{F}}
\newcommand{\Q}{\mathbb{Q}}
\newcommand{\R}{\mathbb{R}}
\newcommand{\Z}{\mathbb{Z}}
\newcommand{\allone}{\mathbf{1}}
\newcommand{\inp}[1]{( #1 )}
\newcommand{\sr}{ r } %{(0,0,2\alpha_1)}

\title[Sets of equiangular lines in dimension $18$]{Sets of equiangular lines in dimension $18$ constructed from $A_9 \oplus A_9 \oplus A_1$}

\begin{document}
\begin{abstract}
	In 2023, Greaves et~al.\ constructed several sets of 57 equiangular lines in dimension 18.
	Using the concept of switching root introduced by Cao et~al.\ in 2021, these sets of equiangular lines are embedded in a lattice of rank 19 spanned by norm 3 vectors together with a switching root.
	We characterize this lattice as an overlattice of the root lattice $A_9\oplus A_9\oplus A_1$, and show that there are at least $246896$ sets of 57 equiangular lines in dimension $18$ arising in this way, up to isometry.
	Additionally, we prove that all of these sets of equiangular lines are strongly maximal.
	Here, a set of equiangular lines is said to be strongly maximal if there is no set of equiangular lines properly containing it even if the dimension of the underlying space is increased.
	Among these sets, there are ones with only six distinct Seidel eigenvalues.
\end{abstract}
\maketitle

% \tableofcontents

\section{Introduction}
A set of lines through the origin in a real vector space is \emph{equiangular} if any pair of these lines forms the same angle.
The maximum cardinality $N(d)$ of a set of equiangular lines in dimension $d$ has been studied since 1948~\cite{Haantjes1948}.
The values or bounds of $N(d)$ for $d \leq 23$ are summarized in Table~\ref{table:N(d)}~\cite{Greaves2022, Greaves2020, Greaves2023, Haantjes1948, lemmens1973, Lint1966}.
\begin{table}[hbtp]
	\caption{The values or bounds of $N(d)$ for $d \leq 23$.}
	\label{table:N(d)}
	\centering
	\begin{tabular}{c|ccccccccccccccc}
		%\hline
		$d$    & 2 & 3 & 4 & 5  & 6  & 7--14 & 15 & 16 & 17 & 18     & 19     & 20     & 21  & 22  & 23  \\
		\hline
		$N(d)$ & 3 & 6 & 6 & 10 & 16 & 28    & 36 & 40 & 48 & 57--59 & 72--74 & 90--94 & 126 & 176 & 276 \\
		%\hline
	\end{tabular}
\end{table}
In any dimension, a general upper bound $N(d) \leq d(d+1)/2$~\cite{lemmens1973} and lower bounds~\cite{de2000, GKMS} are known.
For a fixed common angle $\arccos(\alpha)$, the maximum cardinality $N_{\alpha}(d)$ of a set of equiangular lines with common angle $\arccos(\alpha)$ in dimension $d$ has also been studied.
It is known that $1/\alpha$ must be an odd integer if $N_\alpha(d) > 2d$~\cite[Theorem~3.4]{lemmens1973}.
In the case of $\alpha = 1/3$, sets of equiangular lines with common angle $\arccos(1/3)$ have been enumerated~\cite{Szollosi2018, yoshino2023}, and their maximality has been determined~\cite{cao2021}.
In the case of $\alpha = 1/5$, the so-called Lemmens-Seidel conjecture
$
	N_{1/5}(d) = \max\left\{
	276,
	\left\lfloor (3d-3)/2 \right\rfloor
	\right\}
$
for $d \geq 23$ has been proved~\cite{cao2022, lemmens1973, yoshino2023}.
In high dimensions, Jiang~et~al.\ proved that for every integer $k \geq 2$,
$N_{1/(2k-1)}(d) = \lfloor k(d-1)/(k-1) \rfloor$
for all sufficiently large $d$~\cite{Zhao2021}.

As mentioned above, determining $ N(d) $ and $ N_\alpha(d) $ has attracted significant attention, and a natural question arises as to whether existing sets of equiangular lines can be extended to construct new ones.
To address this, the concept of \emph{strongly maximal} sets of equiangular lines was introduced~\cite{cao2021}.
A set of equiangular lines is said to be strongly maximal if there is no set of equiangular lines properly containing it even if the dimension of the underlying space is increased.
Cao~et~al.~\cite{cao2021} proved that any set of equiangular lines achieving the so-called absolute bound is strongly maximal and determined strongly maximal sets of equiangular lines with common angle $\arccos(1/3)$.

The main focus of this paper is on sets of equiangular lines in dimension $18$, which is the smallest dimension where $N(d)$ remains unknown.
From the relative bound~\cite{lemmens1973}, it follows that $N(18) = N_{1/5}(18) \leq 61$.
Greaves~et~al.~\cite{Greaves2022} further refined this bound, proving that $N(18) \leq 59$.
Also, Taylor remarked that $N(18) \geq 48$~\cite[p.~124]{Daniel1971}.
Sz\H{o}ll\"osi~\cite{Szollosi2019} improved the lower bound to $N(18) \geq 54$.
Subsequently, Lin and Yu~\cite{Lin2020b} proved $N(18) \geq 56$, and show that no extra line can be added to either of the two sets of equiangular lines in~\cite{Lin2020b, Szollosi2019}.
Their approach to constructing large sets of equiangular lines involved removing specific lines from a set of $276$ equiangular lines in dimension $23$, which is known to be unique~\cite[Theorem~A]{goethals1975}.
More recently, Greaves~et~al.~\cite{Greaves2023} proved that $N(18) \geq 57$ by searching for vectors in $\Z^{18}$ of norm $10$ with pairwise inner product $\pm 2$.
They explicitly provided four Seidel matrices corresponding to four sets of $57$ equiangular lines in dimension $18$ and claimed to have identified numerous others.
They then posed the question of whether their four sets of equiangular lines are contained in the set of $276$ equiangular lines in dimension $23$.
Yoshino~\cite{yoshino2022} resolved this question in the negative, and moreover proved that the four sets are strongly maximal.

In this paper, we show that there are a large number of strongly maximal sets of $57$ equiangular lines in dimension $18$ which are realized by sets of norm $3$ vectors in a certain overlattice $\Lambda$ of the root lattice $A_9\oplus A_9\oplus A_1$.
We note the lattice $\Lambda$ in Theorem~\ref{thm:main} is chosen in such a way that it realizes four sets of $57$ equiangular lines in~\cite{Greaves2023}.
In fact, we verify in 
Appendix
%Section~\ref{sec:Greaves} 
that the sets of equiangular lines we found include those explicitly given in~\cite{Greaves2023}.
Our main theorem is as follows. The definition of affine equiangular sets will be given in Definition~\ref{dfn:aes}. It makes possible to represent a set of equiangular lines with common angle $\arccos(1/5)$ in terms of a set of norm $3$ vectors in an integral lattice.
\begin{thm} \label{thm:main}
	There are at least $246896$ affine equiangular sets of cardinality $57$ with norm $3$ with respect to $(0,0,(1,-1))$ in the lattice $\Lambda$ up to isometry.
	Furthermore, all of these affine equiangular sets are strongly maximal.
\end{thm}
Since affine equiangular sets essentially correspond to sets of equiangular lines, as will be introduced and explained in Section~\ref{sec:aes}, this theorem immediately implies the following corollary. 
\begin{cor}	\label{cor:main}
	There are at least $246896$ sets of $57$ equiangular lines in dimension $18$ with common angle $\arccos(1/5)$ up to isometry.
\end{cor}
However, all the sets of equiangular lines we found are strongly maximal as well, and as a result, our result does not directly contribute to a possible of improvement the lower bound on $N(18)$.
It may well be true that $N(18)$ is $57$.

The main tool in our study is the concept of switching roots, introduced in~\cite{cao2021}, which significantly simplifies the search for vectors that form sets of equiangular lines.
A natural approach to searching for vectors that form a set of equiangular lines with common angle of $\arccos(1/5)$ is to look for vectors with norm $5$ and pairwise inner products $\pm 1$ in a given integral lattice.
However, enumerating all such vectors with norm $5$ in a given lattice is often computationally intricate, making it difficult to directly investigate sets of equiangular lines.
Furthermore, although strong maximality can be verified by enumerating vectors of norm at most $5$ in the dual of the lattice, this process is computationally even harder.
By introducing switching roots, these vectors can be translated into an affine space along the direction of the switching root, thereby reducing the problem to searching for vectors of norm $3$ with pairwise inner products $0$ or $1$ in an integral lattice.
Since the resulting lattice often has a smaller discriminant, we can more easily address the problem of whether a set of equiangular lines is strongly maximal.
See Section~\ref{sec:aes} for details.
In this paper, we define such vectors of norm $3$ as an \emph{affine equiangular set} in Definition~\ref{dfn:aes}, and adopt the overlattice $\Lambda$ of $A_9 \oplus A_9 \oplus A_1$ in Definition~\ref{dfn:Lambda} as a suitable lattice.
%In fact, in Corollary~\ref{cor:num}, we provide all affine equiangular sets of cardinality $57$ in the lattice $\Lambda$, which induce sets of $57$ equiangular lines.

This paper is organized as follows.
In Section~\ref{sec:aes}, we introduce affine equiangular sets in connection with sets of equiangular lines, and describe how to verify their strong maximality.
In Section~\ref{sec:Lambda}, we introduce the overlattice $\Lambda$, and in Section~\ref{sec:3}, we provide a set $X$ of vectors of norm $3$ that contains the desired affine equiangular sets.
Since $X$ is partitioned as $X_1 \cup X_{-1} \cup X_5$, and every affine equiangular set in $X$ is switching equivalent to one in $X_1 \cup X_5$,
we study maximum affine equiangular sets in $X_5$ and $X_1$ in Section~\ref{sec:X5} and Section~\ref{sec:X1}, respectively.
In Section~\ref{sec:gen}, we study the lattices coming from maximum affine equiangular sets in $X$ to later consider affine equiangular sets up to isometry.
In Section~\ref{sec:X}, we provide a lower bound on the number of maximum affine equiangular sets in $X$ up to switching, and a lower bound on the number up to isometry in Corollary~\ref{cor:num}.
In addition, we verify that the maximum affine equiangular sets in $X$ are strongly maximal in Theorem~\ref{thm:strongly maximal}, and obtain Theorem~\ref{thm:main}.
In Section~\ref{sec:Howell}, we provide examples sets of $57$ equiangular lines in dimension $18$, whose Seidel matrices have only $6$ distinct eigenvalues, one less than what Geaves~et~al.\ found~\cite{Greaves2023}. 
In Appendix, %\ref{sec:Greaves}
we verify that the sets of equiangular lines constructed in this paper include those explicitly given in~\cite{Greaves2023}.

\section{Sets of equiangular lines and affine equiangular sets}	\label{sec:aes}
In this section, we introduce sets of equiangular lines and affine equiangular sets, and describe the relations between them.
In general, an isometry is defined as a mapping between metric spaces that preserves metric.
In this paper, we treat lines and vectors in  a Euclidean space, and hence isometries are compositions of translations and orthogonal transformations.
Furthermore, we will restrict our attention to isometries which fix certain points, such as the origin or a specific vector (e.g., a switching root), and consequently preserve inner products.
For simplicity, we also refer to a function defined on a set of vectors that preserves inner products as an isometry.

First, we introduce concepts related to sets of equiangular lines.
For a set $\Phi$ of equiangular lines, denote by $\langle \Phi \rangle_{\R}$ the smallest vector subspace containing all lines in $\Phi$.
Two sets $\Phi$ and $\Phi'$ of equiangular lines are said to be \emph{isometric} if there exists an isometry $f : \langle \Phi \rangle_{\R} \to \langle \Phi' \rangle_{\R}$ such that $f(\Phi) = \Phi'$.
Here, note that $f(\Phi)$ denotes $\{ \{ f(u) : u \in l \} : l \in \Phi \}$.
A set $\Phi$ of equiangular lines is said to be \emph{strongly maximal} if there is no set of equiangular lines properly containing $\Phi$ even if the dimension of the underlying space is increased.
Also, a \emph{Seidel matrix} is a symmetric matrix with diagonal entries $0$ and off-diagonal entries $\pm 1$.
Two Seidel matrices are said to be \emph{switching equivalent} if they are similar under some signed permutation matrix.
For a non-negative integer $n$, denote by $[n]$ the set $\{ 1, \ldots, n \}$.
For a set $\Phi = \{ l_i \}_{i=1}^n$ of $n$ equiangular lines with common angle $\arccos(\alpha)$, there exist unit vectors $v_1, \ldots, v_n$ such that $l_i = \R v_i$ for every $i \in [n]$.
Then, the matrix $(1/\alpha)( I- G )$ is a Seidel matrix, where $G$ is the Gram matrix of $v_1, \ldots, v_n$.
Although $v_j$ can be replaced with $-v_j$ as a unit vector generating $l_j$ for some $j$, this Seidel matrix is uniquely determined up to switching equivalence.
In particular, the characteristic polynomial of the Seidel matrix $S$ is an invariant of the set $\Phi$ of equiangular lines.

Next we introduce the concept of affine equiangular sets.
Note that a vector $v$ is called a \emph{root} if $(r,r) = 2$.
\begin{dfn}	\label{dfn:aes}
	Let $s$ be a real number at least $1$.
	Let $r$ be a root in $\R^N$, where $N \in \Z_{\geq 1}$.
	The set of vectors $u_1,\ldots,u_n$ in $\R^N$ is called an \emph{affine equiangular set with norm $s$ with respect to $r$} if
	\begin{enumerate}
		\item $(u_1,r) = \cdots = (u_n,r) = 1$,
		\item $(u_i, u_j) = \begin{cases}
			s   & \text{ if } i = j, \\
			0,1 & \text{ otherwise.}
		\end{cases}$
	\end{enumerate}
	%For simplicity, we will often refer to this as an affine equiangular set.
	The root $r$ is called the \emph{switching root} of the affine equiangular set.
\end{dfn}

For an affine equiangular set $\aes=\{u_1,\ldots,u_n\}$ with norm $s$ with respect to a switching root $r$, we have
\begin{align*}
	\left(\frac{2u_i - r}{\sqrt{2}}, \frac{2u_j - r}{\sqrt{2}}\right)
	= \begin{cases}
		  2s - 1 & \text{ if } i = j, \\
		  -1,1   & \text{ otherwise.}
	  \end{cases}
\end{align*}
Furthermore, the lines $\R(2u_1-r), \ldots, \R(2u_n-r)$ are equiangular with common angle $\arccos(1/(2s-1))$ in the vector space $\langle u_1,\ldots,u_n \rangle \cap r^\perp$.
We say that this set of equiangular lines is \emph{induced} by $\aes$.
Any set of equiangular lines with common angle $\arccos(1/(2s-1))$ is induced by some affine equiangular set with norm $s$.
Refer to~\cite{cao2022} for details.

\begin{dfn}
	Two affine equiangular sets $\aes$ and $\aes'$ %with the same norm 
	with respect to a common switching root $r$ are said to be \emph{switching equivalent}
	if there exists a bijection $g : \aes \to \aes'$ such that either $g(u) = u$ or $g(u) = r-u$ for all $u \in \aes$.
	Denote by $\sim_{\sw}$ the switching equivalence, that is, $\aes \sim_{\sw} \aes'$ if $\aes$ and $\aes'$ are switching equivalent.
	Denote by $[\aes]$ the switching class $\{ \aes' : \aes' \sim_{\sw} \aes \}$.
\end{dfn}

We see that if two affine equiangular sets are switching equivalent, then they induce the same set of equiangular lines.
It should be noted that the concepts above have been essentially introduced in terms of matrices in~\cite{cao2021}.

\begin{dfn}
	Affine equiangular sets $\aes$ with respect to a switching root $r$ and $\aes'$ with respect to a switching root $r'$ are \emph{isometric}
	if there exists an affine equiangular set $\aes''$ switching equivalent to $\aes'$ and a bijective isometry $g : \aes \cup \{r \} \to \aes'' \cup \{r' \}$ such that $g(r) = r'$.
\end{dfn}

\begin{lem}	\label{lem:iso}
	Affine equiangular sets $\aes$ with respect to a switching root $r$ and $\aes'$ with respect to a switching root $r'$ are isometric
	if and only if the set of equiangular lines induced by $\aes$ and that induced by $\aes'$ are isometric.
\end{lem}
\begin{proof}
	Let $\Phi$ and $\Phi'$ be the sets of equiangular lines induced by $\aes$ and $\aes'$, respectively.
	Assume that $\aes$ and $\aes'$ are isometric.
	Then, there exists an affine equiangular set $\aes''$ switching equivalent to $\aes'$ and a bijective isometry $g : \aes \cup \{r \} \to \aes'' \cup \{r' \}$ such that $g(r) = r'$.
	This induces the bijection
	\begin{align*}
		\begin{array}{ccc}
			\{ 2u-r : u \in \aes \} & \to     & \{ 2u''-r' : u'' \in \aes'' \} \\
			2u-r                    & \mapsto & 2g(u)-r'.
		\end{array}
	\end{align*}
	Since this preserves inner product,
	%Indeed, for $u,v \in \aes$,
	%\begin{align*}
	%	(2g(u)-r' , 2g(v)-r')
	%	 & = 4(g(u),g(v)) -2 (g(u), g(r)) - 2 (g(r), g(v)) + (g(r),g(r)) \\
	%	 & = 4(u,v)-2(u,r)-2(r,v)+(r,r)                                  \\
	%	 & = (2u-r,2v-r).
	%\end{align*}
	this mapping can be extended to the isometry $$f : \langle 2u-r : u \in \aes \rangle_{\R} \to \langle  2u''-r' : u'' \in \aes'' \rangle_{\R}.$$
	Here, $\langle 2u - r : u \in \aes \rangle_{\R} = \langle \Phi \rangle_{\R}$ holds, and 
	$$\langle 2u'' - r' : u'' \in \aes'' \rangle_{\R} = \langle 2u' - r' : u' \in \aes' \rangle_{\R} = \langle \Phi' \rangle_{\R}$$ holds by $\aes' \sim_{\sw} \aes''$.
	Since this isometry $f$ satisfies $f(\Phi) = \Phi'$ by definition, we conclude that $\Phi$ and $\Phi'$ are isometric.

	Conversely, we assume that $\Phi$ and $\Phi'$ are isometric.
	Then, there exists an isometry $f : \langle \Phi \rangle_{\R} \to \langle \Phi' \rangle_{\R}$ such that $f(\Phi) = \Phi'$.
	Noting that $u \in \aes$ corresponds one-to-one to $\R(2u-r) \in \Phi$, and $u' \in \aes'$ does one-to-one to $\R(2u'-r') \in \Phi'$, we see that for any $u \in \aes$, there uniquely exists $u' \in \aes'$ such that $f(\R(2u-r)) = \R(2u'-r')$.
	Then,
	$
		f(u) \in \{ u' , r' - u' \}
	$
	holds.
	This means that $\aes'' := \{ f(u) : u \in \aes \}$ is switching equivalent to $\aes'$.
	In addition, the mapping $\aes \to \aes'' ; u \mapsto f(u)$ is a bijective isometry.
	Therefore, $\aes$ and $\aes'$ are isometric.
\end{proof}

\begin{dfn}
	An affine equiangular set $\aes$ with respect to a switching root $r$ is said to be \emph{strongly maximal}
	if there is no affine equiangular set with respect to $r$ properly containing $\aes$, even if the dimension of the underlying space is increased.
\end{dfn}

Note that an affine equiangular set is strongly maximal if and only if the set of equiangular lines induced by it is strongly maximal.
Also, for a set $S$ of vectors, denote by $\langle S \rangle$ the lattice generated by $S$, that is, the set of linear combinations of vectors in $S$ with integral coefficients.
Although, in general, an additive subgroup of $\R^n$ is called a \emph{lattice} if it is discrete,
all sets written as $\langle S \rangle$ in this paper are discrete. % https://cseweb.ucsd.edu/classes/wi10/cse206a/lec1.pdf

\begin{lem}	\label{lem:strongly maximal}
	An affine equiangular set $\aes$ with norm $s$ with respect to $r$ is strongly maximal
	if and only if
	there is no vector $v$ in the dual of the lattice $\langle \aes \cup \{r\} \rangle$ such that $(v,v) \le s$, $(v,r) = 1$ and $(v, u) \in \{0,1\}$ for every $u \in \aes$.
\end{lem}
\begin{proof}
	The affine equiangular set $\aes$ is not strongly maximal if and only if there exists a vector $w$ in some vector space containing $\langle \aes \cup \{r\} \rangle$ such that $(w,r) = 1$ and $(w,u) \in \{0,1\}$ for every $u \in \aes$.
	Let $v$ be the orthogonal projection of such a vector $w$ onto the vector space generated by the vectors of $\aes$ together with $r$.
	Then, $v$ is contained in the dual lattice of $\langle \aes \cup \{r\} \rangle$,
	and satisfies $(w,u) \in \{0,1\}$ for all $u \in \aes$ and $(w,r) = 1$.
	Thus, we obtain the necessary and sufficient condition.
\end{proof}

\section{An overlattice of $A_9\oplus A_9\oplus A_1$}	\label{sec:Lambda}
In this section, we introduce an overlattice of $A_9\oplus A_9\oplus A_1$, which contains affine equiangular sets of cardinality $57$, and describe its properties.
Before that, we first describe properties of the root lattices $A_n$ and their dual.
For detailed explanations, see \cite{Conway1999, Ozeki2018}.
Let $\allone$ be the vector all of whose entries are $1$.
The root lattice $A_n$ is defined as
\[A_n=\{x\in\Z^{n+1} : \inp{\allone,x}=0\}.\]
Its dual quotient $A_n^*/A_n$ is a cyclic group of order $n+1$
generated by the coset of
\[\alpha_n=\left(\frac{1}{n+1},\dots,\frac{1}{n+1},-\frac{n}{n+1}\right).\]
In particular,
\begin{align}
	\disc A_n      & =n+1,\label{1A}               \\
	\|\alpha_n\|^2 & =\frac{n}{n+1}.\label{1alpha}
\end{align}
Also,
\begin{align}	\label{AutA}
	\Aut(A_n) = \langle -1 \rangle \times S_{n+1},
\end{align}
where $S_{n+1}$ is the symmetric group of all permutations of the coordinates.
\begin{lem}[\cite{Ozeki2018}] \label{lem:minA}
	For each $m \in [n+1]$, the $\binom{n+1}{m}$ minimum representatives of $m \alpha_n + A_n \in A_n^*/A_n$ are given by
	\begin{align}	\label{enum_min}
		\frac{m}{n+1} \sum_{i \in [n+1] \setminus I} e_i-\frac{n+1-m}{n+1} \sum_{i \in I} e_i \qquad ( I \in \binom{[n+1]}{m} )
	\end{align}
	with norm
	\begin{align} \label{min}
		\frac{m(n+1-m)}{n+1}.
	\end{align}
\end{lem}

Next, we introduce an overlattice of $A_9\oplus A_9\oplus A_1$, and characterize it.
\begin{dfn}	\label{dfn:Lambda}
	Let $L:=A_9\oplus A_9\oplus A_1$,
	and $\alpha := (\alpha_9, 2\alpha_9, \alpha_1)$.
	Define
	\begin{equation}\label{1L}
		\Lambda:=L+\Z\alpha.
	\end{equation}
	Additionally, let $r := (0,0,2\alpha_1)$.
\end{dfn}

\begin{lem}\label{lem:1}
	The lattice
	$\Lambda$ is the unique integral lattice satisfying the following
	conditions, up to $\Aut(L)$.
	\begin{align}
		 & L\subseteq\Lambda\subseteq L^*,\label{1a}                            \\
		 & \disc\Lambda=2,\label{1b}                                            \\
		 & \text{the projection of $\Lambda$ to $\Q A_1$ is $A_1^*$.}\label{1c}
	\end{align}
	Moreover,
	\begin{align}
		\Lambda / L \simeq \Z/10\Z. \label{1d}
	\end{align}

\end{lem}
\begin{proof}
	The lattice $L+\Z \alpha$ in~\eqref{1L} satisfies the three conditions~\eqref{1a}--\eqref{1c}.
	Conversely, suppose that $\Lambda$ is an integral lattice satisfying \eqref{1a}--\eqref{1c}.
	Observe, by \cite[p.~4]{E},
	\begin{align*}
		|\Lambda:L|^2 & =
		\left(\frac{\vol(\R^n/L)}{\vol(\R^n/\Lambda)}\right)^2
		=
		\frac{\disc L}{\disc\Lambda}
		\nexteq
		\frac{\disc L}{2}
		\refby{1b}
		\nexteq
		\frac{(\disc A_9)^2\disc A_1}{2}
		\nexteq
		100
		\refby{1A}.
	\end{align*}
	Thus, $\Lambda/L\cong\Z/10\Z$.
	By \eqref{1c}, $\Lambda/L$ is generated by the coset of
	an element $(m_1\alpha_9,m_2\alpha_9,\alpha_1)
		\in A_9^*\oplus A_9^*\oplus A_1^*$, where $m_1,m_2\in\Z$
	and $\gcd(m_1,m_2,5)=1$.
	Since $L$ is integral, we have
	\begin{align*}
		\Z & \ni\|(m_1\alpha_9,m_2\alpha_9,\alpha_1)\|^2
		=
		\frac{9(m_1^2+m_2^2)+5}{10}.
	\end{align*}
	This implies $\{m_1^2,m_2^2\}\equiv\{1,4\}$ or $\{9,6\}$ modulo $10$.
	Since $(\pm \id, \pm \id, \id) \in\Aut(L)$, we may assume $(m_1,m_2)=(1,2)$ or $(3,6)$ up to $\Aut(L)$.
	Since $(3\alpha_9,6\alpha_9,\alpha_1)$ and
	$(\alpha_9,2\alpha_9,\alpha_1)$ generate the same cyclic
	subgroup of $L^*/L$, we obtain \eqref{1L}.
\end{proof}

\begin{lem}	\label{lem:aut}
	The automorphism group $\Aut(\Lambda)$ is $\{ \pm(\sigma_1, \sigma_2,  \pm \id)  :  \sigma_1, \sigma_2 \in S_{10} \} $.
\end{lem}
\begin{proof}
	Let $\tau$ be the automorphism of $L$ defined by $\tau(x,y,z) = (y,x,z)$ for $(x,y,z) \in L = A_9 \oplus A_9 \oplus A_1$.
	Then, we see that
	\begin{align*}
		\Aut(L) = \{ \tau^{\varepsilon} \circ (\pm \sigma_1, \pm \sigma_2,  \pm \id)  : \varepsilon \in \{0,1\}, \sigma_1, \sigma_2 \in S_{10} \}.
	\end{align*}
	Let
	\begin{align*}
		\Gamma := \{ \pm(\sigma_1, \sigma_2,  \pm \id)  :  \sigma_1, \sigma_2 \in S_{10} \}.
	\end{align*}
	First, we verify $\Gamma \subset \Aut(\Lambda)$.
	Since $\Gamma \subset \Aut(L)$, it suffices to verify that $f(\alpha) \in \Lambda$ for any $f \in \Gamma$.
%	From Lemma~\ref{lem:minA}, $\sigma(m \alpha_9) \in m \alpha_9 + A_{9}$ for every $\sigma \in S_{10}$, where $m \in [10]$.
Since $\sigma(\alpha_9)-\alpha_9$ is a root in $A_9$ for any $\sigma\in S_{10}$, we see that 
$S_{10}$ acts trivially on $A_9^*/A_9$.
	Hence, we have $(\sigma_1, \sigma_2, \id) \in \Aut(\Lambda)$ for any $\sigma_1, \sigma_2 \in S_{10}$.
	Also, since $(\id,\id,-\id)\alpha = \alpha - \sr \in \Lambda$, we have $(\id,\id,-\id) \in \Aut(\Lambda)$.
	Additionally, since $-\id \in \Aut(\Lambda)$, we obtain $\Gamma \subset \Aut(\Lambda)$ as desired.
	
	Next, we prove $\Gamma \supset \Aut(\Lambda)$.
Let $\upsilon := (\id, -\id, \id ) \in \Aut(L)$.
%	We prove $\Gamma = \Aut(\Lambda)$.
%	Since it follows that
Then
	\begin{align*}
\Aut(\Lambda)\subset\Aut(L)=\Gamma\cup
\bigcup_{f\in\{\tau, \upsilon, \tau \circ \upsilon\}} f\Gamma.
% / \Gamma = \{ \Gamma, \tau \Gamma, \upsilon \Gamma, (\tau \circ \upsilon) \Gamma\},
	\end{align*}
Thus, it suffices to show $\tau, \upsilon, \tau \circ \upsilon \not\in \Aut(\Lambda)$.
%	We fix $f \in \{ \tau, \upsilon, \tau \circ \upsilon \}$, and then may prove that $f(\alpha) \not\in \Lambda$.
%	Indeed, since
Since
	$$\Lambda / L = \bigcup_{m=0}^{9} \left( m\alpha+L \right)$$ 
	holds by Lemma~\ref{lem:1}, 
	it suffices to prove $f(\alpha) - m \alpha \not\in L$ for any 
	$f\in\{\tau, \upsilon, \tau \circ \upsilon\}$ and 
	$m \in \{0,\ldots,9\}$.
	By direct calculation, we have 
	\begin{align*}
		\tau(\alpha) - m \alpha
		&= \left((2-m)\alpha_9, (1-2m)\alpha_9, (1-m)\alpha_1 \right),\\
		\upsilon(\alpha) - m \alpha
		&= \left((1-m)\alpha_9, (-2-2m)\alpha_9, (1-m)\alpha_1 \right),\\
		(\tau \circ \upsilon)(\alpha) - m \alpha
		&= \left((-2-m)\alpha_9, (1-2m)\alpha_9, (1-m)\alpha_1 \right).
	\end{align*}
	Since $(a \alpha_9, b \alpha_9, c \alpha_1) \in L$ for $a,b,c \in \Z$ if and only if $a \equiv b \equiv 0 \pmod{10}$ and $c \equiv 0 \pmod{2}$,
	we obtain that $f(\alpha) \not\in \Lambda$ as desired.
	Therefore, $\Aut(\Lambda) = \Gamma$.
\end{proof}

\section{Vectors of norm $3$ in $\Lambda$}	\label{sec:3}
We will construct affine equiangular sets in $\Lambda$ with norm $3$ with respect to $\sr$ to obtain sets of equiangular lines with a common angle of $\arccos(1/5)$ in dimension $18$.
In this section, we provide a set $X$ of candidate vectors that can form the desired affine equiangular sets.
Subsequently, we partition the set $X$ into two subsets and show that it suffices to construct affine equiangular sets in these subsets independently.
From now on, write the minimum norm of an element of a set $S$ as $\min S$.

\begin{dfn}
	Let \[X:=\left\{\gamma\in\Lambda :  \inp{\gamma,\gamma}=3,\; \inp{\gamma,\sr}=\frac12\right\},\]
	and
	\begin{align}
		X_m:=X\cap(L+m \alpha) \quad(m\in\Z).
	\end{align}
	For $i, j,  k \in [10]$ with  $i \neq j$, let
	\begin{align*}
		v_1(i,j,k)    & := \left(\frac{1}{10}\allone-e_i,\frac{2}{10}\allone-e_j-e_k,\alpha_1\right), \\
		v_{-1}(i,j,k) & := \sr- v_1(i,j,k).
	\end{align*}
	For $I\in\binom{[10]}{5}$, let
	\begin{align*}
		v_5(I) := \left(\frac{1}{2}\allone-\sum_{i\in I}e_i,0,\alpha_1\right).
	\end{align*}
\end{dfn}

\begin{lem}\label{lem:2}
	We have
	\begin{align}
		X=X_1\cup X_{-1}\cup X_5.	\label{lem:2:1}
	\end{align}
	Also,
	\begin{align}
		X_\varepsilon & = \{v_{\varepsilon}(i,j,k) :  i,j,k \in [10],\ j\neq k\}  \quad (\varepsilon \in \{\pm 1\}),	\label{lem:2:2} \\
		X_5           & = \left\{ v_{5}(I)  :  I \in \binom{[10]}{5} \right\}. \label{lem:2:3}
	\end{align}
\end{lem}
\begin{proof}
	Since $(\gamma, \sr) = 1$  for every $\gamma \in X$, we see that $X_m = \emptyset$ for any even integer $m$.
	By Lemma~\ref{lem:minA},
	we have
	\begin{align*}
		\min\left( m\alpha + L \right)
		 & \geq \min \left( m\alpha_9+A_9 \right) + \min \left( 2m\alpha_9+A_9 \right) + \min \left( \alpha_1+A_1 \right) \\
		 & = \begin{cases}
			     5 & \text{ if } m = \pm 3,    \\
			     3 & \text{ if } m = \pm 1, 5.
		     \end{cases}
	\end{align*}
	Hence, $X_{\pm 3} = \emptyset$, and~\eqref{lem:2:1} follows.
	Furthermore, we see that for every vector $\gamma = (\gamma_1, \gamma_2, \gamma_3) \in X$, $\gamma_1$, $\gamma_2$ and $\gamma_3$ are minimal representatives in $m\alpha_9 + A_9$, $2m\alpha_9 + A_9$ and $\alpha_1+A_1$, respectively.
	By Lemma~\ref{lem:minA} again, we have~\eqref{lem:2:2} and~\eqref{lem:2:3}.
\end{proof}

By direct calculation, we have the following.
\begin{lem} \label{lem:inner prod}
	Let $i,j,k \in [10]$ with $j \neq k$, and  $i',j',k' \in [10]$ with $j' \neq k'$.
	Then,
	\begin{align}
		(v_{\varepsilon}(i,j,k), v_{\varepsilon}(i',j',k')) & =
		|\{i\} \cap \{i'\} | + | \{j,k\} \cap \{j',k'\} |\quad(\varepsilon=\pm 1). \label{lem:inner prod:1}
	\end{align}
	Let $I,J \in \binom{[10]}{5}$. Then,
	\begin{align}
		(v_{ 1}(i,j,k), v_5(I)) & =  | \{i\} \cap I |,\label{lem:inner prod:2} \\
		(v_5(I),v_5(J))         & = | I \cap J|-2.\label{lem:inner prod:3}
	\end{align}
	%	where double-signs correspond.
\end{lem}

\begin{dfn}
	Let $\cX$, $\cX_{1}$ and $\cX_5$ be the sets of maximum affine equiangular sets in $X$,
	%		$X_{-1} \cup X_1$,
	$X_{1}$,
	and $X_5$, respectively, with respect to $\sr$.
\end{dfn}

Our aim is to investigate maximum affine equiangular sets $\aes$ in $X$ with norm $3$ with respect to $\sr$
up to switching equivalence, that is, replacing some of the vectors $u \in \aes$
by $\sr-u$.
From the following lemma,
we may restrict our attention to affine equiangular sets in $X_1$ and ones in $X_5$

\begin{lem}	\label{lem:chi}
	We have
	$\cX/\sim_{\sw}
		= \left\{ [\aesf \cup \aeso] : [\aesf] \in \cX_5/\sim_{\sw}, \aeso \in \cX_1 \right\}.$
	In particular, $\left|\cX/\sim_{\sw} \right| = \left|\cX_5/\sim_{\sw} \right| \left| \cX_1 \right|.$
\end{lem}
\begin{proof}
	Let $\cX_{\pm1}$ be the set of maximum affine equiangular sets in $X_1 \cup X_{-1}$.
	By \eqref{lem:inner prod:2}, we have $(u,v) \in \{0,1\}$ for any $u \in X_1 \cup X_{-1}$ and $v \in X_5$.
	Hence,
	\begin{align*}
		\cX = \{ \aesf \cup \aeso : \aesf \in \cX_5, \aeso \in \cX_{\pm1} \}.
	\end{align*}
	Since $X_{-1} = \{ \sr - u  :  u \in X_1 \}$, we see that for each $\phi \in \cX_{\pm1}$, there exists (a unique) $\phi' \in \cX_1$ with $\phi' \sim_{\sw} \phi.$
	Noting that $X_5$ and $X_{-1} \cup X_1$ are closed under switching equivalence $\sim_{\sw}$, we have
	\begin{align*}
		\cX/\sim_{\sw}
		 & = \left\{ [\aesf \cup \aeso] : [\aesf] \in \cX_5/\sim_{\sw}, [\aeso] \in \cX_{\pm1} /\sim_{\sw} \right\} \\
		 & = \left\{ [\aesf \cup \aeso] : [\aesf] \in \cX_5/\sim_{\sw}, \aeso \in \cX_1 \right\}.
	\end{align*}
	Next, we verify $\left|\cX/\sim_{\sw} \right| = \left|\cX_5/\sim_{\sw} \right| \left| \cX_1 \right|.$
	Let $\aesf, \aesf' \in \cX_5$ and $\aeso, \aeso' \in \cX_1$ be such that $[\aesf \cup \aeso] = [\aesf' \cup \aeso']$.
	Then, it suffices to show $[\aesf] = [\aesf']$ and $\aeso = \aeso'$.
	Indeed, we obtain $[\aesf] = [\aesf']$ and $[\aeso] = [\aeso']$.
	Moreover, as discussed earlier, since an element of $\cX_1$ uniquely exists as a representative of $[\aeso]$, we conclude $\aeso = \aeso'$.	
\end{proof}

\section{Maximum affine equiangular sets in $X_5$}	\label{sec:X5}
In this section, we determine a unique maximum affine equiangular set in $X_5$ up to $\Aut(\Lambda)_{\sr}$ with a computer.

\begin{lem}	\label{lem:12}
	Every affine equiangular set in $X_5$ with respect to $\sr$ is of cardinality at most $12$.
\end{lem}
\begin{proof}
	Recall that the vectors in $X_5$ are
	\begin{align*}
		\left(\frac{1}{2}\allone-\sum_{i\in I}e_i,0,\alpha_1 \right) \quad \text{ for } \quad I\in\binom{[10]}{5}.
	\end{align*}
	Hence, the dimension of the linear space generated by $X_5$ together with $\sr$ is at most $\dim(\R^{10} \perp \allone) + 1 = 10$.
	Thus, an arbitrary affine equiangular set in $X_5$ with respect to $\sr$ induces a set of equiangular lines with common angle $\arccos(1/5)$ in dimension $9$.
	The cardinality $N_{1/5}(9)$ of a maximum set of equiangular lines with common angle $\arccos(1/5)$ in dimension $9$ equals $12$~\cite[Theorem~6.7 (with a computer)]{GKMS}.
	Therefore, the affine equiangular set is of cardinality at most $12$.
\end{proof}

\begin{dfn}	\label{dfn:S4}
	Let $S_4$ be the symmetric group on $[4]$.
	Let $c_1,\ldots,c_6$ be the conjugates
	$
		(1, 2, 3, 4)$, $
		(1, 3, 2, 4)$, $
		(1, 2, 4, 3)$, $
		(1, 3, 4, 2)$, $
		(1, 4, 2, 3)$, $
		(1, 4, 3, 2)
	$
	of the cyclic permutation $(1,2,3,4)$ in $S_4$,
	respectively.
	Let $\cI_0$ be the set of
	\begin{align} 	\label{dfn:S4:1}
		I_{i,j} :=  \left( [4] \setminus \{i\} \right) \sqcup \{ k \in [10] \setminus [4] : i^{c_{k-4}} = j \}
		\quad (i,j \in [4],\;i\neq j ).
	\end{align}
	Let $\aesf_0$ be the set of $v_5(I)$ for all $I$ in $\cI_0$.
	Let $\Lambda_5$ be the lattice generated by the vectors in $\aesf_0$ and $\sr$,
	and $L_5$ be the root sublattice of $\Lambda_5$,
	that is, the lattice generated by roots in $\Lambda_5$.
\end{dfn}

The sets~\eqref{dfn:S4:1} are explicitly enumerated as follows.
\begin{align*}
	%		\{
	 & \{ 2, 3, 4, 5, 7 \},
	\{ 2, 3, 4, 6, 8 \},
	\{ 2, 3, 4, 9, 10 \},    \\
	 & \{ 1, 3, 4, 8, 10 \},
	\{ 1, 3, 4, 5, 9 \},
	\{ 1, 3, 4, 6, 7 \},     \\
	 & \{ 1, 2, 4, 7, 9 \},
	\{ 1, 2, 4, 6, 10 \},
	\{ 1, 2, 4, 5, 8 \},     \\
	 & \{ 1, 2, 3, 5, 6 \},
	\{ 1, 2, 3, 8, 9 \},
	\{ 1, 2, 3, 7, 10 \}.
	%    		\}
\end{align*}
The automorphism group $\Aut(\cI_0) \subset S_{10}$ of $\cI_0$ is defined as the bijections $\sigma$ from $[10]$ to itself such that $\sigma(I) \in \cI_0$ for any $I \in \cI_0$.
Also, the symmetric group $S_4$ on $[4]$ acts on $\cI_0$ in such a way that $\sigma(I_{i,j}) := I_{i^\sigma, j^\sigma}$ for $\sigma \in S_4$.
For example, $(1,2) \in S_4$ induces the automorphism 
\begin{align}	\label{Aut(I_0)}
	(1,2)(5,8)(6,9)(7,10) \in \Aut(\cI_0).
\end{align}

\begin{lem}	\label{lem:X50}
	The set $\aesf_0$ is a maximum affine equiangular set in $X_5$ with respect to $\sr$.
	In particular, an affine equiangular set in $X_5$ is maximum if and only if its cardinality is $12$.
\end{lem}
\begin{proof}
	By direct calculation, $\aesf_0$ is an affine equiangular set in $X_5$ with respect to $\sr$.
	In addition, since $\aesf_0$ is of cardinality $12$, it is maximum by Lemma~\ref{lem:12}.
\end{proof}

To prove the following lemmas, we use the software Magma. %, which handles mathematical concepts such as graphs, lattices and groups.
\begin{lem}
	The number $|\cX_5/\sim_{\sw}|$ of maximum affine equiangular sets in $X_5$ with respect to $\sr$ up to switching equivalence is $151200$.	\label{lem:X5}
	Also, the automorphism group $\Aut(\Lambda)_{\sr}$ transitively acts on $\cX_5/\sim_{\sw}$.
\end{lem}
\begin{proof}
	First, we explain how to determine $\cX_5/\sim_{\sw}$.
	Using Magma~\cite{Magma}, we construct the graph $G$ with vertex set $V := \{ \{ u,r-u\} : u \in X_5 \}$ such that for two vertices $x$ and $x'$, they are adjacent if $(u,u') \in \{0,1\}$ for any $u \in x$ and $u' \in x'$.
	Then, $|V| = 126$.
	Note that, for each clique $\{ x_i \}_{i \in [k]}$, it holds that for any $u_i \in x_i$ ($i \in [k]$), $\{u_i\}_{i \in [k]}$ is an affine equiangular set in $X_5$ with respect to $r$.
	Furthermore, the clique induces the switching class $[ \{u_i\}_{i \in [k]} ]$.
	Hence, by Lemma~\ref{lem:X50}, $\cX_5/\sim_{\sw}$ corresponds to the set $C$ of all cliques of  cardinality $12$ in $G$.
	Using Magma again, we can construct $C$, and have $|C| = 151200$.
	
	Next, we explain how to verify the automorphism group $\Aut(\Lambda)_{\sr}$ transitively acts on $\cX_5/\sim_{\sw}$.
	Since \(\Aut(\Lambda)_r \) has been determined by Lemma~\ref{lem:aut}, we see that for $[\psi] \in \cX_5/ \sim_{\sw}$,
	\[
		\Aut(\Lambda)_{\sr} [\psi] = \{ (\sigma, \id, \id)  :  \sigma \in S_{10} \} [\psi].
	\]
	Thus, it suffices to consider the induced action of $S_{10}$ on $C$, instead of the action of $ \Aut(\Lambda)_r $ on $\cX_5 / \sim_{\sw}$.
	Using Magma, we can confirm that $S_{10}$ transitively acts on the maximum cliques.
\end{proof}

We verified the following using Magma by explicitly constructing $L_5$.
\begin{lem}	\label{lem:L5}
	We have
	\begin{align*}
		L_5 & = \langle (e_i-e_j,0,0) : i,j \in [4] \rangle \oplus \langle (e_i-e_j ,0,0): i,j \in [10]\setminus[4] \rangle \oplus \langle \sr \rangle \\
		    & \simeq A_3 \oplus A_5 \oplus A_1.
	\end{align*}
\end{lem}

\section{Maximum affine equiangular sets in $X_1$}	\label{sec:X1}
In this section, we investigate affine equiangular sets in $X_1$ with respect to $\sr$.
We denote by $K_n$ the complete graph with vertex set $[n]$.
The number $T(n)$ of matchings in $K_n$ is called an \emph{involution number} (also, \emph{telephone number}).
This satisfies $T(0)=1$ and $T(n) = T(n-1) + (n-1)T(n-2)$ for $n \geq 1$
(see \cite[p.~85, Problem 17(d)]{MR1949650}).
The values of involution numbers starting from $n=0$ to $10$ are $1, 1, 2, 4, 10, 26, 76, 232, 764, 2620, 9496$. % \ldots$.%(OEIS: A000085)
Also, denote by $F(n)$ the number of $1$-factorization of the complete graph $K_{2n}$, where isomorphic $1$-factorizations are distinguished if their matchings are not equal as sets.
The values of $F(n)$ starting from $n=0$ to $5$ are $1$, $1$, $6$, $6240$, $1225566720$, $252282619805368320$~\cite{DGM94}. %$98758655816833727741338583040$ (OEIS: A000438).

\begin{dfn}
	Let $n$ be a non-negative integer, and
	\begin{align*}
		\cY_n := \left\{ (Y_i)_{i=1}^{2n} : Y_1, \ldots, Y_{2n} \text{ are pairwise disjoint matchings of $K_{2n}$ and } \bigcup_{i=1}^{2n} Y_i = E(K_{2n})  \right\}.
	\end{align*}
\end{dfn}
The cardinality of $\cY_n$ is the number of proper edge colorings of $K_{2n}$ with $2n$ colors.
Since this is not known, we provide a lower bound as follows.

\begin{lem}	\label{lem:XZY}
	For a non-negative integer $n$, $|\cY_n| \geq (2n-1)! F(n)T(2n)$.
\end{lem}
\begin{proof}
	Let
	\begin{align*}
		\cF := \{ (M_i)_{i=1}^{2n-1} : M_1,\ldots, M_{2n-1} \text{ form a  $1$-factorization of } K_{2n} \}.
	\end{align*}
	Let $\mathcal{M}$ be the set of matchings of $K_{2n}$.
	Since $|\cF| = (2n-1)!F(n)$ and $|\mathcal{M}| = T(2n)$,
	the proof will be complete if 
	\begin{align*}
		\begin{array}{rccc}
			f : & \cF \times \mathcal{M}                    & \to     & \cY_n                                                         \\
			    & \left( (M_i)_{i=1}^{2n-1}, M_{2n} \right) & \mapsto & \left( (M_i \setminus M_{2n} )_{i=1}^{2n-1}, M_{2n} \right)
		\end{array}
	\end{align*}
	is injective.
	Let $(M_i)_{i=1}^{2n-1}, (N_i)_{i=1}^{2n-1} \in \cF$ and $M_{2n}, N_{2n} \in \mathcal{M}$.
	Then, assume $(M_i \setminus M_{2n})_{i=1}^{2n-1} = (N_i \setminus N_{2n})_{i=1}^{2n-1}$ and $M_{2n} = N_{2n}$.
	To prove $M_i = N_i$ for every $i$ by way of contradiction, we assume that $M_i \neq N_i$ for some $i$.
	There exists an edge $e := \{ x, y \} \in M_i \setminus N_i$.
	Then, $e \in M_{2n} = N_{2n}$ follows.
	Since $N_i$ is a $1$-factor, there exists an edge $e' := \{ x, z \} \in N_i$  for some $z$.
	Since $M_i$ is a $1$-factor, $e' \not\in M_i$ holds, and hence $e' \in N_{2n} = M_{2n}$ follows.
	Thus, $x \in e \cap e'$ and $e , e' \in N_{2n}$.
	However, this contradicts the fact that the assumption that $N_{2n}$ is a matching.
\end{proof}

Next, we establish a one-to-one correspondence between $\cX_1$ and $\cY_5$ in the following lemma.
\begin{lem}	\label{lem:X1Y10}
	The maximum affine equiangular sets $\aeso \in \cX_1$ in $X_1$ with respect to $\sr$ correspond one-to-one to the elements $(Y_i)_{i=1}^{10} \in \cY_{5}$ in such a way that
	\begin{align}	\label{lem:X1Y10:0}
		( \{ \{j,k\} \in E(K_{10}) : v_1(i,j,k) \in \aeso \} )_{i=1}^{10}  % \quad \text{ for each } \quad \aes \in \cX_1,
	\end{align}
	in $\cY_5$, and
	\begin{align}	\label{lem:X1Y10:1}
		\left\{ v_1(i,j,k) : \{j,k\} \in Y_i \ \text{ for }\ i \in [10] \right\}%  \quad \text{ for each } \quad (Y_i)_{i=1}^{10} \in \cY_{5}.
	\end{align}
	in $\cX_1$.
	In particular, a maximum affine equiangular set in $X_1$ with respect to $\sr$ is of cardinality $45$.
\end{lem}
\begin{proof}
	Recall from Lemma~\ref{lem:inner prod} that for \(1 \leq i, i', j \neq k, j' \neq k' \leq 10\),
	\begin{align*}
		(v_1(i, j, k), v_1(i', j', k')) = |\{i\} \cap \{i'\}| + |\{j, k\} \cap \{j', k'\}|.
	\end{align*}
	This implies that the affine equiangular sets $\aeso$ in \(X_1\) with respect to \(\sr\) correspond one-to-one to the tuples of pairwise disjoint \(10\) matchings \((Y_i)_{i=1}^{10}\) in \(K_{10}\),
	where \(\aeso\) and \((Y_i)_{i=1}^{10}\) are related to each other by~\eqref{lem:X1Y10:0} and~\eqref{lem:X1Y10:1}, respectively.
	Hence, for an arbitrary affine equiangular set $\aeso$,
	\[
		|\aeso| = \sum_{i=1}^{10} |Y_i| \leq |E(K_{10})| = 45,
	\]
	with equality if and only if $(Y_i)_{i=1}^{10} \in \cY_{5}$.
\end{proof}

\begin{lem}	\label{lem:num1}
	We have
	$|\cX_1| \geq 9!T(10)F(5) = 21502468915200.$
\end{lem}
\begin{proof}
	By Lemma~\ref{lem:X1Y10}, $|\cX_1| = |\cY_{5}|$.
	Also, by Lemma~\ref{lem:XZY}, $|\cY_5| \geq 9! T(10) F(5)$.
	Here, $F(5) = 6240$ and $T(10) = 9496$.
	Thus, the desired conclusion follows.
\end{proof}

\section{The Lattices generated by maximum affine equiangular sets in $X$}	\label{sec:gen}

In this section, our aim is to prove that the lattice generated by a maximum affine equiangular set in $X$ together with the switching root $\sr$ coincides with $\Lambda$.
\begin{lem}	\label{lem:con}
	Let $(Y_i)_{i=1}^{10} \in \cY_5$.
	Let
	\begin{align}
		E_+ := \bigcup_{i=1}^4 Y_i \quad \text{and} \quad E_- := \bigcup_{i=5}^{10} Y_i.
		\label{lem:con:Epm}
	\end{align}
	Let $G$ be a graph with vertex set $[10]$ and edges $\{j,j'\} \in \binom{[10]}{2}$ satisfying for some $k \in [10]$,
	\begin{align}	\label{lem:con:0}
		\{\{j,k\}, \{j',k\}\} \subset E_+ \ \text{ or }\ \{\{j,k\}, \{j',k\}\} \subset E_-.
	\end{align}
	Then, $G$ is connected.
\end{lem}
\begin{proof}
	For any $j \in [10]$, define $N_{\pm}(j) := \{ h\in[10]:\{j,h\} \in E_{\pm}\}$.
	By way of contradiction, we assume that $G$ is not connected.
	Then, there are two vertices $j$ and $j'$ which are not contained in the same connected component.
	Since $Y_1, \ldots, Y_{4}$ are matchings, we have
	\begin{align*}
		|N_+(j)| \leq 4
		\quad \text{and} \quad
		|N_+(j')| \leq 4
	\end{align*}
	hold.
	Also,
	\begin{align}	\label{lem:con:1}
		N_+(j) \cap N_+(j') = \emptyset
	\end{align}
	and
	\begin{align}	\label{lem:con:2}
		\{j,j'\} \cup N_+(j) \cup N_+(j') = [10]
	\end{align}
	hold.
	Indeed, if~\eqref{lem:con:1} or~\eqref{lem:con:2} does not hold,
	then there is $k \in [10]$ such that~\eqref{lem:con:0}, that is, $j$ and $j'$ are adjacent in $G$.
	This is a contradiction.
	Thus, we see that
	\begin{align*}
		|N_+(j)| = |N_+(j)| = 4
		\quad \text{and} \quad
		\{j,j'\} \sqcup N_+(j) \sqcup N_+(j') = [10].
	\end{align*}
	In other words,
	\begin{align}
		|N_-(j)| = |N_-(j')| = 5
		\quad \text{and} \quad
		N_-(j) \sqcup N_-(j') = [10].
	\end{align}
	In particular, $N_-(j)$ and $N_-(j')$ are connected component of $G$.
	Then, we have
	\begin{align}	\label{lem:con:-1}
		E_- = \left\{\{ l, l' \} : l \in N_-(j),\; l' \in N_-(j') \right\}
		\quad \text{and} \quad
		E_+ = \binom{N_-(j)}{2} \cup  \binom{N_-(j')}{2}.
	\end{align}
	Indeed, if $h,k\in N_-(j)$ and $\{h,k\}\in E_-$, then
	$\{\{j,k\},\{h,k\}\}\subset E_-$, and hence $\{j,h\}$ is an edge of $G$.
	Since $j\in N_-(j')$ and $h\in N_-(j)$, this is impossible.
	Thus, we have shown $\binom{N_-(j)}{2}\subset E_+$.
	Similarly, we obtain $\binom{N_-(j')}{2}\subset E_+$. Since $|E_+|=\sum_{i=1}^4|Y_i|\leq20$,
	equality is forced, and we conclude that \eqref{lem:con:-1} holds.

	%	Noting $|E_+| = 20$, we notice that 
	Now $Y_1, \ldots, Y_4$ are perfect matchings.
	However, since $E_+$ forms the disjoint union of $K_5$ and $K_5$,
	$E_+$ does not contain perfect matchings.
	This is a contradiction.
\end{proof}

\begin{lem}	\label{lem:genL}
	Every maximum affine equiangular set in $X$ with respect to $\sr$ and the switching root $\sr$ generates $\Lambda$.
\end{lem}
\begin{proof}
	Let $\aes' \in \cX$.
	Let $M'$ be the lattice generated by $\aes'$ together with $r$.
	By Lemma~\ref{lem:chi}, we may assume without loss of generality that
	$\aes' = \aesf' \cup \aeso'$ for some
	% there exist 
	$\aesf' \in \cX_5$ and $\aeso' \in \cX_1$.
	By Lemma~\ref{lem:X5}, there exists $\sigma \in S_{10}\times\{\id \}\times\{ \id \}\subset
		\Aut(\Lambda)$ such that
	\begin{align}	\label{lem:genL:1}
		\aesf := \sigma(\aesf') = \psi_0.
	\end{align}
	Let $M:=\sigma(M')$. Then,
	Lemma~\ref{lem:L5} together with~\eqref{lem:genL:1} implies
	\begin{align} \label{lem:genL:3}
		(A_3 \oplus A_5 ) \oplus O_{10} \oplus A_1  \subset M.
	\end{align}
	Here, denote by $O_n$ the lattice $\{ 0 \} \subset \Z^n$ for a positive integer $n$.
	Let $\aeso := \sigma(\aeso')$ and $\aes:=\aesf\cup\aeso$.
	Then, $\aes$ together with $r$ generates $M$.
	Moreover, since $\sigma$ leaves $X_1$ invariant, we have $\aeso \in \cX_1$ .
	%	by Lemma~\ref{lem:aut}.
	By Lemma~\ref{lem:X1Y10}, there exists $(Y_i)_{i=1}^{10} \in \cY_{5}$ such that
	\begin{align} \label{lem:genL:4}
		\aeso = \left\{ v_1(i,j,k) : \{j,k\} \in Y_i \ \text{ for }\  i \in [10] \right\} \subset X_1.
	\end{align}
	Let $E_+$ and $E_-$ be as in~\eqref{lem:con:Epm}.
	%	$E_+ := \bigcup_{i=1}^{4} Y_i $ and $E_- := \bigcup_{i=5}^{10} Y_i$ as in~\eqref{lem:con:-1}.

	First, we prove
	\begin{align} \label{lem:genL:2}
		O_{10} \oplus A_9 \oplus A_1 \subset M.
	\end{align}
	Indeed, we define a graph $G$ with vertex set $[10]$ and edges $\{j, j'\} \in \binom{[10]}{2}$ if
	\begin{align}	\label{lem:genL:6}
		(0,e_{j'}-e_j,0) \in M.
	\end{align}
	It suffices to prove that $G$ is connected.
	Note that, if $\{j,k\}\in Y_i$ and $\{j',k\}\in Y_{i'}$, then
	%	 for $i,i', j \neq k \neq j' \in [10]$,
	\begin{align}	\label{lem:genL:5}
		v_1(i,j,k)-v_1(i',j',k) = (e_{i'}-e_i, e_{j'}-e_j , 0) \in M.
	\end{align}
	Therefore, we have from~\eqref{lem:genL:3},
	$(0,e_{j'}-e_j,0) \in M$, or equivalently, $\{j,j'\} \in E(G)$,
	provided that \eqref{lem:con:0} holds.
	By applying Lemma~\ref{lem:con} to $G$, we conclude that $G$ is connected, that is, \eqref{lem:genL:2} holds.

	Next, we prove
	\begin{align}	\label{lem:genL:7}
		A_9 \oplus A_9 \oplus A_1  \subset M.
	\end{align}
	There exist $i \in [4]$ and $i' \in [10] \setminus [4]$ such that $Y_{i'}$ and $Y_i$ are non-empty.
	Fix $\{ j',k' \} \in Y_{i'}$ and $\{j,k\} \in Y_{i}$.
	Then,
	\begin{align*}
		v_1(i,j,k)-v_1(i',j',k') = (e_{i'} - e_i , e_{j'} + e_{k'} - e_j - e_k , 0) \in M.
	\end{align*}
	By~\eqref{lem:genL:2}, we have
	$
		(e_{i'} - e_i , 0 , 0) \in M.
	$
	By~\eqref{lem:genL:3} and~\eqref{lem:genL:2}, we obtain~\eqref{lem:genL:7}.

	Finally, we prove $M=\Lambda$.
	The elements $\aes \cap X_5$ are of order $2$ in $\Lambda/L$, and the elements $\aes \cap X_1$ are of order $5$ in $\Lambda/L$.
	Since $L \subset M$ by~\eqref{lem:genL:7} and $|\Lambda/L|=10$ by~\eqref{1d},
	we conclude $M= \Lambda$.
\end{proof}

\section{Maximum affine equiangular sets in $X$ and strong maximality}	\label{sec:X}
In this section, we estimate the number of maximum affine equiangular sets in $X$ up to switching in Theorem~\ref{thm:num}, and up to isometry in Corollary~\ref{cor:num}.
This corollary implies Theorem~\ref{thm:main}.
Furthermore, we verify that these sets are strongly maximal in Theorem~\ref{thm:strongly maximal}.

\begin{thm}	\label{thm:num}
	The number $|\cX/ \sim_{\sw}|$ of maximum affine equiangular sets in $X$ with respect to $\sr$ up to switching equivalence is at least $3251173299978240000$.
\end{thm}
\begin{proof}
	By Lemma~\ref{lem:chi}, we have
	$
		|\cX/\sim_{\sw}| = |\cX_5/\sim_{\sw}| \cdot |\cX_1|.
	$
	Since $|\cX_5/\sim_{\sw}| = 151200$ by Lemma~\ref{lem:X5} and $|\cX_1| \geq 21502468915200$ by Lemma~\ref{lem:num1}, we have the desired conclusion.
\end{proof}

The stabilizer subgroup $\Aut(\Lambda)_{r}$ of $\Aut(\Lambda)$ with respect to the switching root $r$ naturally acts on $\cX/\sim_{\sw}$.
Specifically, for $f \in \Aut(\Lambda)_r$ and $[\aes] \in \cX/\sim_{\sw}$, the action is defined as $f([\aes]) := [\{ f(u) : u \in \aes \}]$.

\begin{cor}	\label{cor:num}
	The number of affine equiangular set of cardinality $57$ with norm $3$ with respect to $\sr$ in $\Lambda$ is at least $246896$ up to isometry.
\end{cor}
\begin{proof}
	Define a relation $\sim_{\iso}$ on $\cX$ in such a way that for $\aes, \aes' \in \cX$, $\aes \sim_{\iso} \aes'$ if and only if $\aes$ and $\aes'$ are isometric.
	For $\aes \in \cX$, denote by $[\aes]_{\iso}$ the isometry class $\{ \aes' \in \cX : \aes' \sim_{\iso} \aes\}$ of $\aes$ in $\cX$ under $\sim_{\iso}$.
	Then the mapping
	\begin{align*}
		\begin{array}{rccc}
			h : & \Aut(\Lambda)_r \backslash \left( \cX / \sim_{\sw} \right) & \to     & \cX/ \sim_{\iso} \\
			    & \Aut(\Lambda)_r [\aes]                                     & \mapsto & [\aes]_{\iso}.
		\end{array}
	\end{align*}
	is well-defined and surjective.
	Next, we prove the mapping $h$ is injective.
	Let $\aes$ and $\aes'$ be elements of $\cX$ such that $[\aes]_{\iso} = [\aes']_{\iso}$.
	Then, there exist $\aes'' \in \cX$ with $\aes' \sim_{\sw} \aes''$ and a bijective isometry $g : \aes \cup \{r\} \to \aes'' \cup \{r \}$ such that $g(r) = r$.
	Since $\langle \aes \cup \{r\} \rangle = \langle \aes'' \cup \{ r \} \rangle = \Lambda$ by Lemma~\ref{lem:genL}, there exists $f \in \Aut(\Lambda)_r$ such that $$g =  f |_{\aes \cup \{ r \}} : \aes \cup \{ r \} \to \aes'' \cup \{ r \}.$$
	Therefore, $f([\aes]) = [f(\aes)] = [\aes''] = [\aes']$ holds, and $\Aut(\lambda)_r[\aes] = \Aut(\lambda)_r[\aes']$ follows.
	This means that $h$ is injective.
	Therefore
	\begin{align}	\label{cor:num:1}
		\left| \cX/ \sim_{\iso} \right|
		= \left| \Aut(\Lambda)_r \backslash \left( \cX / \sim_{\sw} \right) \right|.
	\end{align}
	Also, $\mu := (-\id,-\id, \id)$ is contained in $\Aut(\Lambda)_r$ by Lemma~\ref{lem:aut}.
	Since $\mu(u) = r - u$ for any $u \in X$, we see that $\mu(\aes) \sim_{\sw} \aes$ for any $\aes \in \cX$.
	Thus, the kernel of the action of $\Aut(\Lambda)_{r}$ on $\cX / \sim_{\sw}$ contains the subgroup of order $2$ generated by $\mu$.
	Since $|\Aut(\Lambda)_r| = 2 \cdot (10!)^2$ by Lemma~\ref{lem:aut}, % and $|\cX/\sim_{\sw}| = 3251173299978240000$ by Theorem~\ref{thm:num},
	\eqref{cor:num:1}  implies
	\begin{align}	\label{cor:num:2}
		\left| \cX/ \sim_{\iso} \right|
		\geq \frac{ | \cX / \sim_{\sw} | }{\left| \Aut(\Lambda)_r/ \langle \mu \rangle \right|} 
		%&= \left| \Aut(\Lambda)_r \backslash \left( \cX / \sim_{\sw} \right) \right| \\
		= \frac{ | \cX / \sim_{\sw} | }{(10!)^2} 
		%&\geq 2 \cdot \frac{ | \cX / \sim_{\sw} | }{| \Aut(\Lambda)_r | } \\
		%\frac{3251173299978240000}{(10!)^2} 
		%& = 2 \cdot \frac{3251173299978240000}{2 \cdot (10!)^2} \\
		\geq 246896
	\end{align}
	by Theorem~\ref{thm:num}.
\end{proof}
We remark that in this proof, equality does not hold in the first inequality in~\eqref{cor:num:2} since the quotient group $\Aut(\Lambda)_{\sr}/\langle \mu \rangle$ does not act on $\cX/\sim_{\sw}$ semiregularly.
Refer to Example~\ref{ex:non-trivial} for details. 
Finally, we prove Theorem~\ref{thm:strongly maximal}, which asserts that any $\aes \in \cX$ is strongly maximal.
\begin{lem}	\label{lem:11}
	There is no vector $v$ in $\Lambda^* \setminus \Lambda$ such that $(v,v) \leq 3$ and $(v,\sr) = 1$.
\end{lem}
\begin{proof}
	Let $\beta := (2\alpha_9, -\alpha_9, 0)$.
	We have $\Lambda^* = \Lambda +  \Z \beta$ from~\eqref{1b}.
	Assume, for contradiction, that such a vector $v$ exists.
	Since $\Lambda^* \setminus \Lambda = \beta + \Lambda$, we have by Lemma~\ref{lem:1}
	\begin{align*}
		v \in \Lambda^* \setminus \Lambda
		= \bigcup_{i=0}^{9} ( i \alpha + \beta  + L ).
	\end{align*}
	By the condition $(v,\sr)  = 1$, we have
	\begin{align*}
		v
		\in \bigcup_{i=1,3,5,7,9} ( i \alpha + \beta  + L )
		= \bigcup_{(i,j) \in P} \left((i \alpha_9, j \alpha_9 , \alpha_1) + L \right),
	\end{align*}
	where
	$
		P := \left\{ (3,1),(5,5),(7,9),(9,3),(1,7)  \right\}.
	$
	We have
	\begin{align*}
		(v,v) \geq \min_{(i,j) \in P} \big( \min\left(i \alpha_9 + A_9 \right)+ \min\left( j \alpha_9 + A_9 \right) + \min\left( a_1 + A_1 \right) \big).
	\end{align*}
	By~\eqref{min}, we have
	\begin{align*}
		(v,v) \geq \frac{1}{10} \cdot \min \left\{
		21+9,
		25+25
		%21+9,
		%9+21,
		%9+21
		\right\} + \frac{1}{2}
		= 3 + \frac{1}{2}.
	\end{align*}
	This contradicts $(v,v) \leq 3$.
\end{proof}

\begin{thm}	\label{thm:strongly maximal}
	Every maximal affine equiangular set contained in $X$
	%	 with norm $3$ 
	with respect to $\sr$ is strongly maximal.
\end{thm}
\begin{proof}
	This follows from Lemmas~\ref{lem:strongly maximal}, \ref{lem:genL} and~\ref{lem:11}.
\end{proof}

\section{Examples}	\label{sec:Howell}
In this section, we first provide examples of $1$-factorizations.  
We construct, from one of the $1$-factorizations, sets of $57$ equiangular lines in dimension $18$ whose Seidel matrix has exactly $6$ eigenvalues, one less than what Greaves~et~al.\ found~\cite{Greaves2023}.  
Subsequently, we examine the automorphism group of an affine equiangular set, and verify that the quotient group $\Aut(\Lambda)_{\sr}/\langle \mu \rangle$ does not act on $\cX/\sim_{\sw}$ semiregularly.
%\begin{align}	\label{iso-sw}
%	|\cX/\sim_{\iso}| > \frac{|\cX/\sim_{\sw}|}{\left| \Aut(\Lambda)_r/ \langle \mu \rangle \right|}
%\end{align} 

\begin{example}
%	As mentioned above, 
	We provide a $1$-factorization of the complete graph $K_{2n}$ for a positive integer $n$ as follows~\cite[5.46]{handbook} .
	For $k \in [2n-1]$, 
	\begin{align*}
		H_k := \left\{ \{ 2n, k \} \right\} \cup \left\{ \{ (k-i) \bmod{2n-1}  , (k+i) \bmod{2n-1} \} : i \in [2n-1]  \right\} \subset \binom{[2n]}{2}
	\end{align*}
	is a $1$-factor, where $(k \pm i) \bmod{2n-1} \in [2n-1]$.
	Then, the set $\{ H_i \}_{i=1}^{2n-1}$ is a $1$-factorization.
	In the case of $n = 5$, we have
	\begin{gather*}
		H_1 := \{\{10, 1\}, \{9, 2\}, \{8, 3\}, \{7, 4\}, \{6, 5\}\}, \\
		H_2 := \{\{10, 2\}, \{1, 3\}, \{9, 4\}, \{8, 5\}, \{7, 6\}\}, \\
		H_3 := \{\{10, 3\}, \{2, 4\}, \{1, 5\}, \{9, 6\}, \{8, 7\}\}, \\
		H_4 := \{\{10, 4\}, \{3, 5\}, \{2, 6\}, \{1, 7\}, \{9, 8\}\}, \\
		H_5 := \{\{10, 5\}, \{4, 6\}, \{3, 7\}, \{2, 8\}, \{1, 9\}\}, \\
		H_6 := \{\{10, 6\}, \{5, 7\}, \{4, 8\}, \{3, 9\}, \{2, 1\}\}, \\
		H_7 := \{\{10, 7\}, \{6, 8\}, \{5, 9\}, \{4, 1\}, \{3, 2\}\}, \\
		H_8 := \{\{10, 8\}, \{7, 9\}, \{6, 1\}, \{5, 2\}, \{4, 3\}\}, \\
		H_9 := \{\{10, 9\}, \{8, 1\}, \{7, 2\}, \{6, 3\}, \{5, 4\}\}.
	\end{gather*}
\end{example}
%Next, we construct sets of $57$ equiangular lines in dimension $18$ whose Seidel matrices have $6$ eigenvalues.
To investigate the properties of a given set of equiangular lines, the eigenvalues of Seidel matrix induced by it have been considered~\cite{cao2021, Greaves2023, yoshino2021}.
In particular, it is often thought that Seidel matrices with fewer eigenvalues tend to exhibit better properties.
Greaves~et~al.\ found sets of $57$ equiangular lines in dimension $18$ whose Seidel matrices have at least $7$ distinct eigenvalues in~\cite{Greaves2023}. 
We found ones with $6$ distinct eigenvalues as in the following example.

\begin{example}	\label{ex:Seidel spec}
	Let $\aeso_0$ and $\aeso_1$ be the maximum affine equiangular sets in $\cX_1$ induced by 
	$$(H_1, H_2, \ldots,H_9,\emptyset) \in \cY_5 \quad \text{ and } \quad (\emptyset, H_1, \ldots,H_8, H_9) \in \cY_5,$$
	respectively, by Lemma~\ref{lem:X1Y10}.
	The characteristic polynomial of the Seidel matrices induced by $\aesf_0 \sqcup \aeso_0$ and $\aesf_0 \sqcup  \aeso_1$ are
	\begin{align}	\label{ex:Seidel spec:1}
		(x-5)^{39}(x+9)^2(x+11)^{11}(x+13)^3(x^2+17x+36)
	\end{align}
	and 
	\begin{align}	\label{ex:Seidel spec:2}
		(x-5)^{39}(x+9)(x+11)^{13}(x+13)^2(x^2+17x+32),
	\end{align}
	respectively.	
\end{example}

\begin{remark}
In addition to the two sets of equiangular lines in Example~\ref{ex:Seidel spec}, there exist many other sets of equiangular lines whose Seidel matrices have the same characteristic polynomials as~\eqref{ex:Seidel spec:1} or~\eqref{ex:Seidel spec:2}.  
Below, we describe our method for finding affine equiangular sets corresponding to such sets of equiangular lines.  
First, since there are $396$ 1-factorizations of $K_{10}$ up to permutation~\cite[5.26]{handbook}, we fix a $1$-factorization among them sequentially.
Next, we randomly constructed $100$ elements $(Y_i)_{i=1}^{10}$ of $\cY_5$ such that for some $j \in [10]$, $Y_j = \emptyset$ and $\{ Y_i \}_{i \in [10] \setminus \{j\}}$ forms the chosen $1$-factorization of $K_{10}$.  
Subsequently, we constructed the corresponding affine equiangular sets $\aeso \in \cX_1$ using Lemma~\ref{lem:X1Y10},  
and then we found that, based on our computer experiments, the maximum affine equiangular sets $\aesf_0 \sqcup \aeso$ have the desired properties without exception.
Here, $\aesf_0 \in \cX_5$ has been given in Definition~\ref{dfn:S4}.
That is, the corresponding Seidel matrices have the characteristic polynomial in~\eqref{ex:Seidel spec:1} or~\eqref{ex:Seidel spec:2}.
\end{remark}

Next, we verify that the action of $\Aut(\Lambda)_{\sr}/\langle \mu \rangle$ on $\cX/\sim_{\sw}$ is not semiregular.  
This suggests that, in order to determine $|\cX/\sim_{\iso}|$, it is necessary to examine the cardinalities of the orbits of $\cX/\sim_{\sw}$ under the action of $\Aut(\Lambda)_{\sr}/\langle \mu \rangle$ in more detail.  
Before providing an example in $\cX/\sim_{\sw}$ with respect to which the stabilizer is nontrivial, we introduce an additional definition.

\begin{dfn}
	Let $n$ be a positive integer.
	The automorphism group $\Aut(\{ F_i \}_{i=1}^{2n-1})$ of $1$-factorization $\{ F_i \}_{i=1}^{2n-1}$ of $K_{2n}$ is the set of bijections from $[2n]$ to itself such that $f(F_i) \in \{ F_i \}_{i=1}^{2n-1}$ for any $i \in[2n-1]$.
\end{dfn}

\begin{example}
As asserted in~\cite[5.55]{handbook}, the automorphism group of the $1$-factorization $\{H_i\}_{i=1}^9$ is of order $54$.
In fact, this group is the semidirect product of the cyclic group generated by the permutation $(1,2,3,4,5,6,7,8,9)$ and the group $(\Z/9\Z)^\times$, where $[9]$ and $\Z/9\Z$ are identified.
%This is the subgroup of $S_{10}$ generated by the elements $(1,2,3,4,5,6,7,8,9)$ of order $9$ and $(1,2,4,8,7,5)(3,6)$ of order $6$, which are additive in $\Z/9\Z$.
For example, $-1 \in (\Z/9\Z)^{\times}$ induces the element 
\begin{align}	\label{Aut(HM)}
	(1,8)(2,7)(3,6)(4,5) \in \Aut(\{ H_i \}_{i=1}^9).
\end{align}
\end{example}

Next, we explicitly construct an example in $\cX/\sim_{\sw}$ with respect to which the stabilizer of $\Aut(\Lambda)_{\sr}/\langle \mu \rangle$ is nontrivial.
\begin{example}	\label{ex:non-trivial}
	Using Lemma~\ref{lem:num1}, we define $\aeso \in \cX_1$ corresponding to
	\begin{align*}
		(H_1,H_8,H_9,\emptyset, H_2,H_3,H_4,H_7,H_6,H_5 ) \in \cY_{5},
	\end{align*}
	and set $\aes := \aesf_0 \sqcup \aeso$.
	Then, $\aes \in \cX$ by Lemma~\ref{lem:chi}.
	Also, let $$\sigma := \left( (1,2)(5,8)(6,9)(7,10) , (1,8)(2,7)(3,6)(4,5), \id \right) \in \Aut(\Lambda)_{r}.$$
	Then, since $\sigma(\aesf_0) = \aesf_0$ follows from~\eqref{Aut(I_0)} and $\sigma(\aeso) = \aeso$ holds from~\eqref{Aut(HM)},
	we obtain $\sigma(\aesf_0 \sqcup \aeso) = \aesf_0 \sqcup \aeso$.
	Therefore, the stabilizer subgroup with respect to $[\omega] \in \cX/\sim_{\sw}$ of $\Aut(\Lambda)_{\sr}/\langle \mu \rangle$ contains $\sigma \langle \mu \rangle$, and is nontrivial.
\end{example}

\section*{Acknowledgements}
%The work of Akihiro Munemasa was supported by JSPS KAKENHI Grant Number 20K03527.
%The research of Tetsuji Taniguchi was supported by JSPS KAKENHI Grant Number 21K03344.
We are grateful to Jack Koolen for his valuable discussions.
\bibliographystyle{plain}
\bibliography{A9A9A1.bib}

\begin{thebibliography}{10}

\bibitem{Magma}
W.~Bosma, J.~Cannon, and C.~Playoust.
\newblock The {M}agma algebra system. {I}. {T}he user language.
\newblock {\em J. Symbolic Comput.}, 24(3-4):235--265, 1997.

\bibitem{cao2022}
M.-Y. Cao, J.H. Koolen, Y.-C.R. Lin, and W.-H. Yu.
\newblock The {L}emmens-{S}eidel conjecture and forbidden subgraphs.
\newblock {\em J. Combin. Theory Ser. A}, 185:105538, 2022.

\bibitem{cao2021}
Meng-Yue Cao, Jack~H. Koolen, Akihiro Munemasa, and K.~Yoshino.
\newblock Maximality of {S}eidel matrices and switching roots of graphs.
\newblock {\em Graphs Combin.}, 37(5):1491--1507, 2021.

\bibitem{handbook}
C.J. Colbourn and J.H. Dinitz.
\newblock {\em Handbook of Combinatorial Designs}.
\newblock Discrete Mathematics and Its Applications. Taylor \& Francis, 2006.

\bibitem{Conway1999}
J.~H. Conway and N.~J.~A. Sloane.
\newblock {\em Sphere packings, lattices and groups}, volume 290 of {\em
  Grundlehren der mathematischen Wissenschaften [Fundamental Principles of
  Mathematical Sciences]}.
\newblock Springer-Verlag, New York, third edition, 1999.

\bibitem{de2000}
D.~de~Caen.
\newblock Large equiangular sets of lines in {E}uclidean space.
\newblock {\em Electron. J. Combin.}, 7:Research Paper 55, 3, 2000.

\bibitem{DGM94}
J.H. Dinitz, D.K. Garnick, and B.D. McKay.
\newblock There are {$526,915,620$} nonisomorphic one-factorizations of
  {$K_{12}$}.
\newblock {\em J. Combin. Des.}, 2(4):273--285, 1994.

\bibitem{E}
W.~Ebeling.
\newblock {\em Lattices and codes}.
\newblock Advanced Lectures in Mathematics. Springer Spektrum, Wiesbaden, third
  edition, 2013.

\bibitem{goethals1975}
J.-M. Goethals and J.~J. Seidel.
\newblock The regular two-graph on {$276$} vertices.
\newblock {\em Discrete Math.}, 12:143--158, 1975.

\bibitem{GKMS}
G.R.W. Greaves, J.H. Koolen, A.~Munemasa, and F.~Sz\"{o}ll\H{o}si.
\newblock Equiangular lines in {E}uclidean spaces.
\newblock {\em J. Combin. Theory Ser. A}, 138:208--235, 2016.

\bibitem{Greaves2022}
G.R.W. Greaves and J.~Syatriadi.
\newblock Real equiangular lines in dimension 18 and the jacobi identity for
  complementary subgraphs.
\newblock {\em Journal of Combinatorial Theory, Series A}, 201:105812, 2024.

\bibitem{Greaves2020}
G.R.W. Greaves, J.~Syatriadi, and P.~Yatsyna.
\newblock Equiangular lines in low dimensional {E}uclidean spaces.
\newblock {\em Combinatorica}, 41(6):839--872, 2021.

\bibitem{Greaves2023}
G.R.W. Greaves, J.~Syatriadi, and P.~Yatsyna.
\newblock Equiangular lines in {E}uclidean spaces: dimensions 17 and 18.
\newblock {\em Math. Comp.}, 92(342):1867--1903, 2023.

\bibitem{Haantjes1948}
J.~Haantjes.
\newblock Equilateral point-sets in elliptic two- and three-dimensional spaces.
\newblock {\em Nieuw Arch. Wiskunde (2)}, 22:355--362, 1948.

\bibitem{Zhao2021}
Z.~Jiang, J.~Tidor, Y.~Yao, S.~Zhang, and Y.~Zhao.
\newblock Equiangular lines with a fixed angle.
\newblock {\em Ann. of Math. (2)}, 194(3):729--743, 2021.

\bibitem{lemmens1973}
P.~W.~H. Lemmens and J.~J. Seidel.
\newblock Equiangular lines.
\newblock {\em J. Algebra}, 24:494--512, 1973.

\bibitem{Lin2020b}
Y.-C.~R. Lin and W.-H. Yu.
\newblock Saturated configuration and new large construction of equiangular
  lines.
\newblock {\em Linear Algebra Appl.}, 588:272--281, 2020.

\bibitem{Ozeki2018}
M.~Ozeki.
\newblock A detailed study of the relationship between some of the root
  lattices and the coding theory.
\newblock {\em Kyushu Journal of Mathematics}, 72(1):123--141, 2018.

\bibitem{MR1949650}
J.~Riordan.
\newblock {\em An introduction to combinatorial analysis: Reprint of the 1958
  original}.
\newblock Dover Publications, Inc., Mineola, NY, 2002.

\bibitem{Szollosi2019}
F.~Sz\"{o}ll\H{o}si.
\newblock A remark on a construction of {D}. {S}. {A}sche.
\newblock {\em Discrete Comput. Geom.}, 61(1):120--122, 2019.

\bibitem{Szollosi2018}
F.~Sz\"{o}ll\H{o}si and P.~\"{O}sterg\aa rd.
\newblock Enumeration of {S}eidel matrices.
\newblock {\em European J. Combin.}, 69:169--184, 2018.

\bibitem{Daniel1971}
D.E. Taylor.
\newblock {\em Some topics in the theory of finite groups}.
\newblock PhD thesis, University of Oxford, 1971.

\bibitem{Lint1966}
J.~H. van Lint and J.~J. Seidel.
\newblock Equilateral point sets in elliptic geometry.
\newblock {\em Nederl. Akad. Wetensch. Proc. Ser. A 69=Indag. Math.},
  28:335--348, 1966.

\bibitem{yoshino2022}
K.~Yoshino.
\newblock The {L}emmens-{S}eidel conjecture for base size $5$.
\newblock {\em arXiv:2209.08308}, 2022.

\bibitem{yoshino2021}
K.~Yoshino.
\newblock Spectral proofs of maximality of some seidel matrices.
\newblock {\em Interdiscip. Inf. Sci.}, 28(1):107--110, 2022.

\bibitem{yoshino2023}
K.~Yoshino.
\newblock Enumeration of sets of equiangular lines with common angle
  $\arccos(1/3)$.
\newblock {\em arXiv:2312.10384}, 2023.

\end{thebibliography}

\section*{Appendix: Sets of equiangular lines found by Greaves et~al.} \label{sec:Greaves}
Greaves et~al.\ provided four sets of $57$ equiangular lines with a common angle $\arccos(1/5)$ in dimension $18$~\cite{Greaves2023}.
These are included among the many sets of $57$ equiangular lines we have constructed.
In this appendix, we explicitly provide the corresponding affine equiangular sets.
To achieve this, we define
\begin{align*}
	\aes(\cI) := \{ v_1(I) : I \in \cI \} \cup \phi_0
\end{align*}
where $\cI$ is a set of triples $(i,j,k)$ and $v_1((i,j,k))$ represents $v_1(i,j,k)$.
We provide four sets $\cI_1, \ldots, \cI_4$ below.
The four affine equiangular sets $\aes(\cI_1), \ldots, \aes(\cI_4)$ induce the four sets of $57$ equiangular lines in dimension $18$.
These sets of equiangular lines correspond to the Seidel matrix $S_i$ coming from $F_i$ in~\cite{Greaves2023} for $i = 1,2,3,4$.
%Here, the switching root is $[0, \cdots, 0 , -1, 1] \in A_2^9A_1$.
For short, we denote $10$ by $0$, and denote $(i,j,k)$ by $ijk$.
Let $\cI_1$ be the set of the following .
\begin{align*}
	1 1 7 ,
	1 2 5 ,
	1 3 0 ,
	1 4 6 ,
	1 8 9 ,
	2 1 3 ,
	2 2 4 ,
	2 5 0 ,
	2 6 9 ,
	2 7 8 ,
	3 1 9 ,
	3 2 3 ,
	3 4 5 ,
	3 8 0 ,
	4 2 6 , \\
	4 3 9 ,
	4 5 8 ,
	4 7 0 ,
	5 1 4 ,
	5 2 8 ,
	5 3 5 ,
	5 6 7 ,
	5 9 0 ,
	6 1 0 ,
	6 2 7 ,
	6 3 4 ,
	6 5 9 ,
	6 6 8 ,
	7 1 6 ,
	7 4 7 , \\
	8 1 8 ,
	8 2 0 ,
	8 3 6 ,
	8 4 9 ,
	8 5 7 ,
	9 1 2 ,
	9 3 8 ,
	9 4 0 ,
	9 5 6 ,
	9 7 9 ,
	0 1 5 ,
	0 2 9 ,
	0 3 7 ,
	0 4 8 ,
	0 6 0 .
\end{align*}
Let $\cI_2$ be the set of the following tuples.
\begin{align*}
	1 2 5 ,
	1 3 0 ,
	1 4 6 ,
	1 7 9 ,
	2 1 9 ,
	2 3 4 ,
	2 5 0 ,
	2 6 8 ,
	3 1 8 ,
	3 2 0 ,
	3 3 6 ,
	3 4 7 ,
	4 1 2 ,
	4 3 8 ,
	4 5 9 , \\
	4 6 7 ,
	5 1 4 ,
	5 2 6 ,
	5 3 5 ,
	5 7 8 ,
	5 9 0 ,
	6 1 0 ,
	6 2 7 ,
	6 4 9 ,
	6 5 8 ,
	7 1 3 ,
	7 2 9 ,
	7 4 8 ,
	7 5 6 ,
	7 7 0 , \\
	8 1 7 ,
	8 2 8 ,
	8 3 9 ,
	8 4 5 ,
	8 6 0 ,
	9 1 5 ,
	9 2 4 ,
	9 3 7 ,
	9 6 9 ,
	9 8 0 ,
	0 1 6 ,
	0 2 3 ,
	0 4 0 ,
	0 5 7 ,
	0 8 9 .
\end{align*}
Let $\cI_3$ be the set of the following tuples.
\begin{align*}
	1 1 9 ,
	1 2 4 ,
	1 3 8 ,
	1 5 6 ,
	1 7 0 ,
	2 1 8 ,
	2 2 5 ,
	2 3 0 ,
	2 4 6 ,
	2 7 9 ,
	3 2 8 ,
	3 4 7 ,
	3 5 9 ,
	4 1 6 ,
	4 2 7 , \\
	4 3 5 ,
	4 4 8 ,
	4 9 0 ,
	5 1 7 ,
	5 3 6 ,
	5 4 0 ,
	5 8 9 ,
	6 1 4 ,
	6 2 0 ,
	6 3 9 ,
	6 5 8 ,
	6 6 7 ,
	7 1 2 ,
	7 3 4 ,
	7 5 7 , \\
	7 6 9 ,
	7 8 0 ,
	8 1 5 ,
	8 2 9 ,
	8 3 7 ,
	8 6 0 ,
	9 1 0 ,
	9 2 3 ,
	9 4 5 ,
	9 6 8 ,
	0 1 3 ,
	0 2 6 ,
	0 4 9 ,
	0 5 0 ,
	0 7 8 .
\end{align*}
Let $\cI_4$ be the set of the following tuples.
\begin{align*}
	1 1 3 ,
	1 2 7 ,
	1 4 5 ,
	1 6 9 ,
	1 8 0 ,
	2 1 6 ,
	2 2 8 ,
	2 3 4 ,
	2 5 7 ,
	2 9 0 ,
	3 1 4 ,
	3 2 0 ,
	3 3 8 ,
	3 5 6 ,
	3 7 9 , \\
	4 1 0 ,
	4 2 3 ,
	4 4 6 ,
	4 5 9 ,
	4 7 8 ,
	5 1 9 ,
	5 2 4 ,
	5 3 0 ,
	5 5 8 ,
	6 1 5 ,
	6 2 9 ,
	6 3 7 ,
	6 4 0 ,
	6 6 8 ,
	7 4 9 , \\
	7 6 7 ,
	8 1 7 ,
	8 2 5 ,
	8 3 9 ,
	8 4 8 ,
	8 6 0 ,
	9 1 8 ,
	9 2 6 ,
	9 3 5 ,
	9 7 0 ,
	0 1 2 ,
	0 3 6 ,
	0 4 7 ,
	0 5 0 ,
	0 8 9 .
\end{align*}

\end{document}